\documentclass[11pt, psamsfonts]{amsart}

\usepackage[a4paper,margin=4cm]{geometry}
\usepackage{amsfonts,amssymb,amscd}
\usepackage{graphicx,mathrsfs} 
\usepackage{subfigure}
\usepackage[utf8]{inputenc}
\newtheorem{theorem}{Theorem}[section]
\newtheorem{lemma}[theorem]{Lemma}
\newtheorem{corollary}[theorem]{Corollary}

\newtheorem{definition}[theorem]{Definition}

\theoremstyle{remark}

\newtheorem{remark}[theorem]{\bf Remark}

\renewcommand{\geq}{\geqslant}

\newcommand{\ptl}{\partial}

\newcommand{\Om}{\Omega}

\newcommand{\rr}{\mathbb{R}}
\newcommand{\sph}{\mathbb{S}}
\newcommand{\nn}{\mathbb{N}}


\numberwithin{equation}{section}

\begin{document}

\title[Cheeger sets for rotationally symmetric planar convex bodies]
{Cheeger sets for rotationally symmetric planar convex bodies}

\author[A. Ca\~nete]{Antonio Ca\~nete}
\address{Departamento de Matem\'atica Aplicada I \\ Instituto de Matemáticas IMUS \\ Universidad de Sevilla}
\email{antonioc@us.es}




\begin{abstract}
In this note we obtain some properties of the Cheeger set $C_\Om$ asociated to a $k$-rotationally symmetric planar convex body $\Om$.
More precisely, we prove that $C_\Om$ is also $k$-rotationally symmetric and touches all the \emph{edges} of $\Om$.
\end{abstract}

\maketitle

\section{Introduction}

The \emph{Cheeger problem} is a classical problem in Geometry widely studied in literature,
with connections in many different fields.
The interested reader may find in~\cite{SS} some historical remarks on closely related questions
(considered, among others, by J.~Steiner or A.~S.~Besicovitch), see also~\cite[Problem~A23]{cfg}.
The origin of this problem is usually established in a paper by J. Cheeger~\cite{cheeger}, who proved in 1969 the following
inequality for any bounded domain $\Om$ in $\rr^n$ (in fact, this result was stated for any compact Riemannian manifold without boundary):
\begin{equation}
\label{eq:original}
\lambda_1(\Om)\geq\frac{h(\Om)^2}{4},
\end{equation}
where $\lambda_1(\Om)$ is the first eigenvalue of the Laplacian on $\Om$ (under Dirichlet conditions),
and
\begin{equation}
\label{eq:CheegerConstant}
h(\Om)=\inf_{X\subseteq\overline\Om}\frac{P(X)}{V(X)},
\end{equation}
where the infimum in~\eqref{eq:CheegerConstant} is taken over all non-empty sets $X$ contained in $\overline\Om$,
and $P(X)$ and $V(X)$ denote the perimeter 
and the volume of $X$, respectively.
This constant $h(\Om)$ is usually called the \emph{Cheeger constant} of $\Om$,
and any subset $X$ contained in $\overline\Om$ providing the infimum in~\eqref{eq:CheegerConstant} is called a \emph{Cheeger set} of $\Om$.
Determining the Cheeger constant and the Cheeger sets of $\Om$, as well as the properties of them, is the main core of the classical Cheeger problem.

This question, intimately related to the classical isoperimetric problem
(see, for instance, the introductory texts~\cite{O,rr}),
has applications in very numerous distinct settings.
As a sample, we briefly enumerate some of them.
It is well known that the Cheeger constant is the limit
of the sequence of first eigenvalues of the $p$\,-Laplacian (with Dirichlet conditions)
of a bounded domain when $p$ tends to 1~\cite{KF}.
This result was later extended to the sequence of second eigenvalues in~\cite{parini-second}
by using the notion of higher Cheeger constants, see~\cite{BP2}.
In the field of image regularization, an approach for removing noise in a given image \cite{ROF}
is strongly connected to the two-dimensional Cheeger problem, as described in~\cite[\S.~2.3]{leonardi}, see also \cite{ACC-interfaces}.
Additionally, the Cheeger constant explicitly appears in the condition on the pressure needed
for breaking down a planar plate~\cite{keller},
and also in some maximal flow-minimal cuts planar problems~\cite{strang}
(with further applications in the field of medical imaging).
Moreover, the Cheeger constant is involved
in the problem of finding hypersurfaces in $\rr^n$ with prescribed (constant) mean curvature,
and provides some equilibrium capillary free-hypersurfaces in certain cases.
The interested reader may find more details of these applications (and some others) in~\cite[\S.~7]{parini} or~\cite[\S.~2]{leonardi}.
We can even find analogues to the Cheeger problem in graph theory~\cite{tan}.

For a given bounded domain $\Om$ in $\rr^n$,
finding the Cheeger constant and the Cheeger sets is, in general, a hard task.
However, questions such as the existence and uniqueness of the Cheeger sets,
as well as intrinsic properties of these sets,
have been well studied in different works.
Our Lemma~\ref{le:properties} collects some of the main results in this direction,
but we remark here that the existence of Cheeger sets is assured in our setting
due to the boundedness of $\Om$, and therefore the infimum in~\eqref{eq:CheegerConstant} is always attained under this hypothesis.
Moreover, if $\Om$ is convex, we have uniqueness of the associated Cheeger set,
which will be also convex with smooth boundary.

In the two-dimensional case, the Cheeger problem is more treatable,
although the complete determination of the Cheeger sets is only known for some specific situations.
Among all the papers in this context, we remark a particular one by B.~Kawohl and T.~Lachand-Robert~\cite{KL},
where we can find some interesting characterizations of the Cheeger sets for planar convex sets.
One of these results establishes a condition for a planar convex set $\Om$
for assuring that the Cheeger set of $\Om$ is $\Om$ itself
(in other words, a condition on $\Om$ to be \emph{calibrable}, see~\cite{ACC-annalen}),
in terms of an inequality involving the curvature of the boundary of $\Om$~\cite[Th.~2]{KL}. 
This result can be applied to the circle, as well as the stadium domain and certain ellipses.
On the other hand,~\cite[Th.~1]{KL} gives a useful description of the Cheeger set of any planar convex set
(specifying also the value of the Cheeger constant) as a Minkowski sum in terms of the inner parallel set
(see also~\cite{B},~\cite[Th.~3.32]{SZ}).
We will use this result, which is stated in Theorem~\ref{th:KL}, in our Section~\ref{sec:main}.

The Cheeger problem in the class of convex polygons is another question of particular interest,
treated for instance in~\cite[\S.~4 and~5]{KL} and in~\cite{BF}.
For a given convex polygon $P$,
standard properties yield that its associated Cheeger set $C_P$ will be bounded by certain line segments and certain circular arcs, see Lemma~\ref{le:properties}.
Taking into account this, it may happen that the boundary $\ptl C_P$ of $C_P$ touches \emph{all} the sides of the polygon $P$.
In that case, $P$ is called \emph{Cheeger-regular} polygon.
Some examples of this situation are triangles and rectangles,
and a necessary and sufficient (analytic and geometrical) condition is established in~\cite[Th.~3]{KL}.
Otherwise, if $\ptl C_P$ does not touch all the sides of $P$, then $P$ is a Cheeger-irregular polygon
(as commented in \cite{KL}, a quadrilateral with a very small side is Cheeger-irregular).
Some remarkable consequences are the following: for a Cheeger-regular convex polygon $P$,
its associated Cheeger set can be obtained by \emph{rounding all the corners} of $P$,
and for any general convex polygon, a direct algorithm
for determining its Cheeger set is described in \cite[\S.~5]{KL}.

In this note, we study the notion of Cheeger-regular set
for the class of \emph{$k$-rotationally symmetric planar convex bodies}.
This class of sets has been considered in the last years 
for some partitioning problems involving the maximum relative diameter functional
(see~\cite{mejora} and references therein).
We will see that we can properly define the concepts of dots and edges
for any $k$-rotationally symmetric planar convex body $\Om$,
somehow in analogy with the vertices and the sides of the polygons.
We will prove in Theorem~\ref{th:main} that any set of this class is Cheeger-regular
(that is, the corresponding Cheeger set $C_\Om$ touches all the edges of $\Om$),
and additionally, we will see in Corollary~\ref{co:rot} that
the $k$-rotational symmetry of $\Om$ is inherited by  $C_\Om$.

{\bf Notation.} Throughout this paper, we will denote by $d$ the Euclidean distance in the plane.
For a given planar bounded set $X$, $P(X)$ will stand for the perimeter of the boundary of $X$
(that is, the one-dimensional Hausdorff measure of $\ptl X$),
and $A(X)$ will stand for the Euclidean area of $X$.
Moreover, the addition of two planar sets must be understood as the classical Minkowski addition.
We also recall that a planar body is, as usual, a planar compact set.

\section{Preliminaries}

\subsection{Some generalities on the Cheeger problem}
As stated in the Introduction, the Cheeger problem has been deeply studied in the literature.
We collect in Lemma~\ref{le:properties} some well-known results, which can be found in many different texts
(see for instance~\cite{AC,leonardi,parini} and references therein).

\begin{lemma}
\label{le:properties}
Let $\Om$ be a planar convex body. Then, there exists a Cheeger set $C_\Om$ of $\Om$. Moreover,
\begin{itemize}
\item[i)] $C_\Om$ is unique, convex with smooth boundary.
\item[ii)] $C_\Om$ touches the boundary of $\Om$.
\item[iii)] $C_\Om$ is an isoperimetric region in $\Om$ for the area it encloses.
\item[iv)] The pieces of the boundary of $C_\Om$ in the interior of $\Om$ are circular arcs of curvature $1/h(\Om)$,
where $h(\Om)$ is the Cheeger constant of $\Om$.
\end{itemize}
\end{lemma}

We can find an interesting characterization of the Cheeger set for any planar convex body $\Om$ in \cite{KL}.
For any $t>0$, denote by $\Om^t$ the inner parallel set of $\Om$ at distance $t$, that is,
$$\Om^t=\{x\in\Om: dist(x,\ptl\Om)>t\}.$$
Then we have the following result.

\begin{theorem}(\cite[Th.~1]{KL})
\label{th:KL}
Let $\Om$ be a planar convex body. Then, there exists a unique value $s>0$ such that $A(\Om^{s})=\pi s^2$.
We also have that $h(\Om)=1/s$ and the Cheeger set of $\Om$ is $C_\Om=\Om^{s}+s\,B_1$, where $B_1$ is the Euclidean unit disk.
\end{theorem}

Theorem~\ref{th:KL} gives a nice geometrical characterization of the Cheeger constant
in this convex setting,
by means of the inner parallel set $\Om^{s}$ of $\Om$ which encloses the same area as the planar disk of radius $s$,
and shows that $C_\Om$ is precisely the Minkowski addition of that inner parallel set and that disk.
From this result we also deduce that the part of $\ptl C_\Om$ contained in the interior of $\Om$
is a union of circular arcs of radius $s$,
and hence $h(\Om)$ (which coincides with $1/s$)
can be identified as the \emph{curvature} of $\ptl C_\Om$ in the interior of $\Om$ (see~\cite[Th.~3.32]{SZ}).

\subsection{The class of $k$-rotationally symmetric planar convex bodies}
Given $k\in\nn$, $k\geq 2$, a planar convex body $\Om$ 
is said to be \emph{$k$-rotationally symmetric}
if there exists a point $p\in\Om$ such that $\Om$ is invariant
under the rotation $\theta_k$ of angle $2\pi/k$ centered at $p$
(the point $p$ is usually called the \emph{center of symmetry} of $\Om$).
Note that, in that case, $\theta_k(\ptl\Om)=\ptl\Om$. 
Some examples of this type of sets are the regular polygons or Reuleaux polygons, see Figure~\ref{fig:examples}.
\begin{figure}[ht]
  \includegraphics[width=0.8\textwidth]{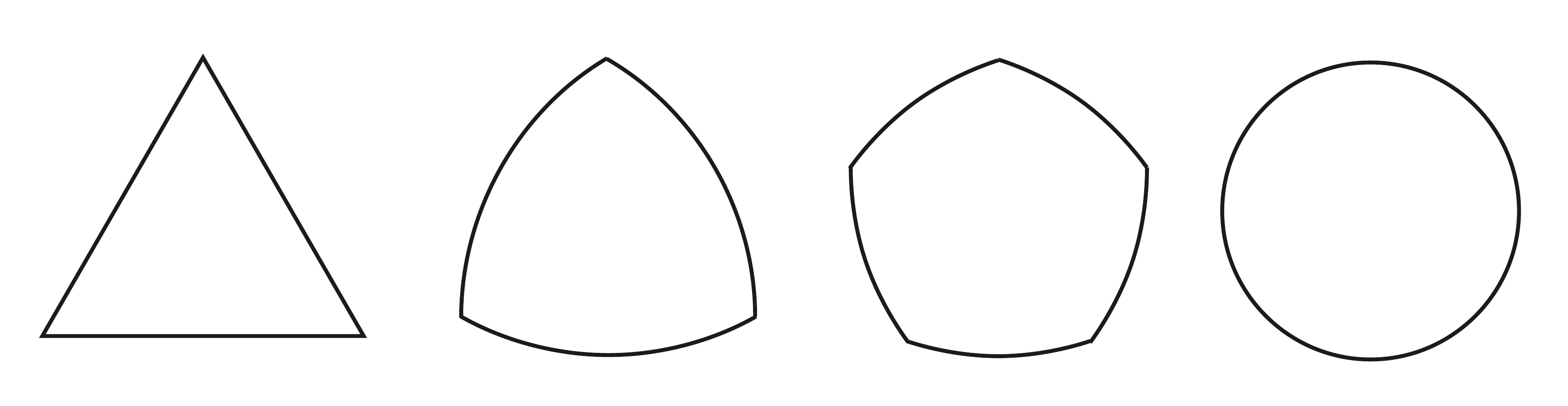}\\
  \caption{Some examples of $k$-rotationally symmetric planar convex bodies: an equilateral triangle and a Reuleaux triangle ($k=3$), a Reuleaux pentagon ($k=5$), and a circle}\label{fig:examples}
\end{figure}

\noindent Figure~\ref{fig:optimals} shows
two interesting $k$-rotationally symmetric planar convex bodies, for $k=3$ and $k=5$.
They are obtained by interesecting the unit disk with a certain equilateral triangle,
and with a certain regular pentagon.
These sets appear as the optimal bodies for an optimization division problem
involving the maximum relative diameter functional~\cite[Th.~5.1]{extending}.
\begin{figure}[ht]
  \includegraphics[width=0.53\textwidth]{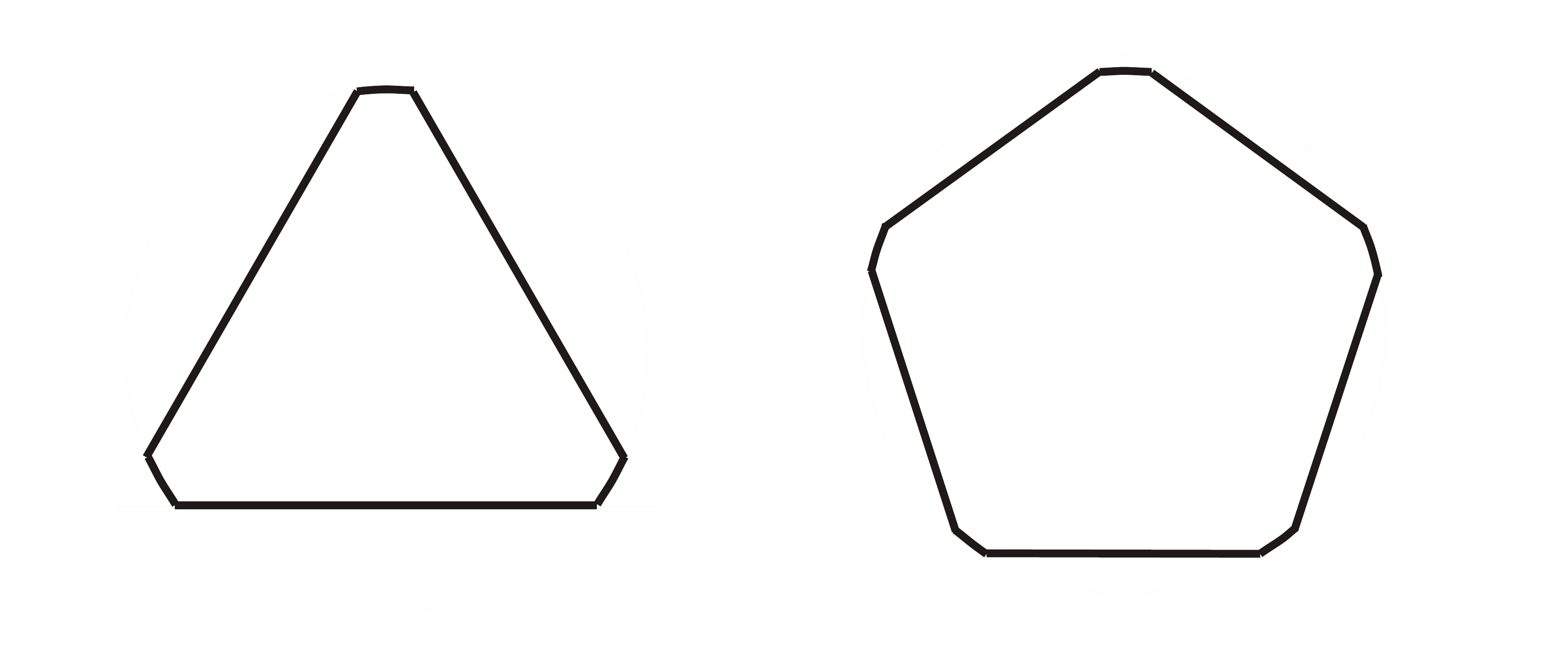}\\
  \caption{Two $k$-rotationally symmetric planar convex bodies, for $k=3$ and $k=5$}\label{fig:optimals}
\end{figure}

\noindent Additionally, the set from Figure~\ref{fig:contraejemplos} is another remarkable example
of this type of sets for $k=2$.
It provides an example where the standard bisection is not minimizing
for the maximum relative diameter functional, see details in~\cite[Ex.~3]{mejora}.
\begin{figure}[ht]
  \includegraphics[width=0.53\textwidth]{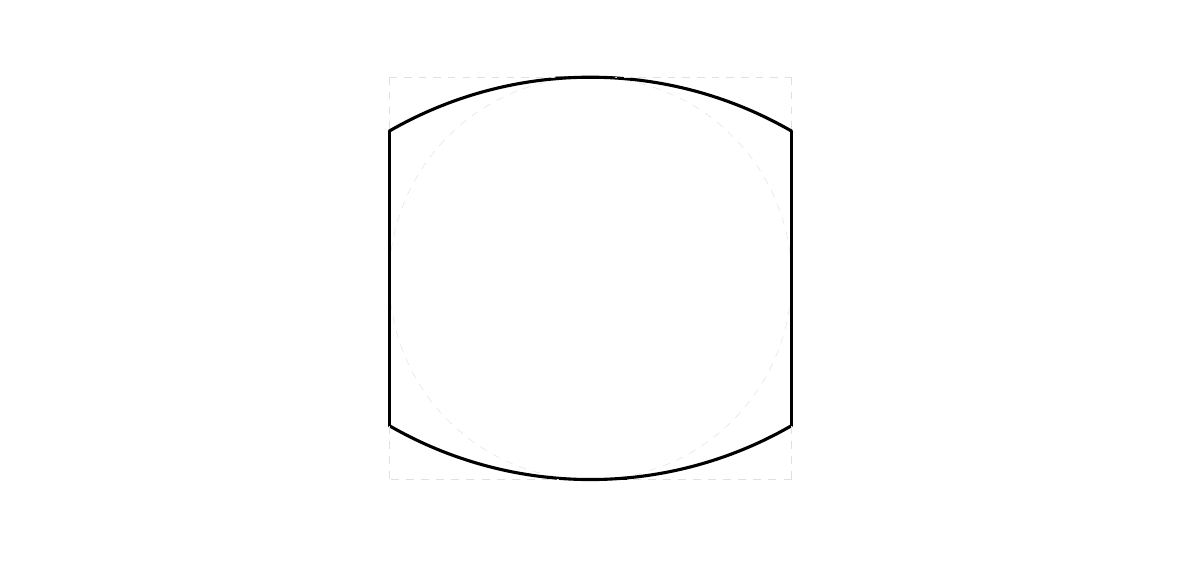}\\
  \caption{A 2-rotationally symmetric planar convex body}\label{fig:contraejemplos}
\end{figure}

We will now define the notions of dots and edges of a $k$-rotationally symmetric planar convex body $\Om$
(keeping certain analogy with the vertices and the sides of polygons).
Let $R>0$ be the circumradius of $\Om$, and let $p$ be the center of symmetry of $\Om$.
Then, there exist $x_1,\ldots, x_k \in \ptl\Om$ such that $d(x_i,p)=R$,
which can be further chosen to be $k$-rotationally symmetric
(observe that the circumradius is necessarily attained by a point $x_1\in\ptl\Om$,
and by applying repeatedly the rotation $\theta_k$ we will obtain the rest of the points).
We will call \emph{dots} of $\Om$ to any choice of these points
(notice that the set of dots may not be unique).
Moreover, for a given set of dots $x_1,\ldots, x_k$ of $\Om$,
an \emph{edge} of $\Om$ will be any piece of $\ptl\Om$ delimited by two consecutive dots.
This implies that $\Om$ will have $k$ edges, namely $L_i=[x_i,x_{i+1}]$, $i=1,\ldots,k$,
with the convention that $x_{k+1}=x_1$, see Figure~\ref{fig:vertices}.
We will consider the edges $L_1,\ldots, L_k$ of $\Om$ in the counterclock-wise order along $\ptl\Om$.
\begin{figure}[ht]
  \includegraphics[width=0.74\textwidth]{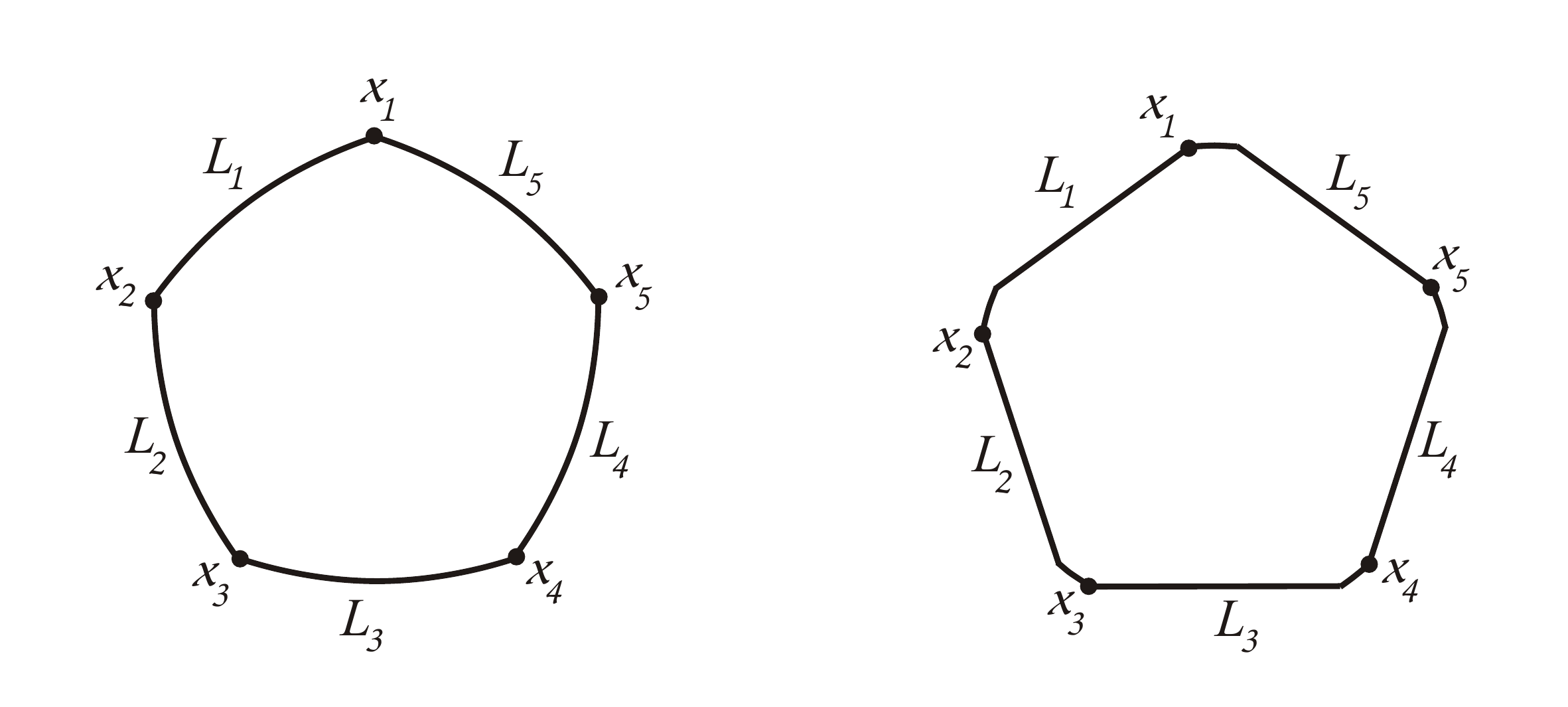}\\
  \caption{Dots and edges of two different $5$-rotationally symmetric planar convex bodies}\label{fig:vertices}
\end{figure}

\begin{remark}
For a given $k$-rotationally symmetric polygon $P$,
the notions of \emph{sides} and \emph{edges} of $P$ do not necessarily coincide.
For instance, a regular polygon $P_n$ of $n$ sides is clearly $n$-rotationally symmetric,
and the sides and edges of $P_n$ will be the same.
However, if we consider an equilateral triangle and we (slightly) cut its vertices symmetrically (see Figure~\ref{fig:cuttri}),
we will obtain a non-regular hexagon 
which is $3$-rotationally symmetric, with six sides and three edges.
The same happens for the \emph{vertices} and the \emph{dots} of a $k$-rotationally symmetric polygon with $n$ sides (for $k\neq n$).
\end{remark}

\begin{figure}[ht]
  \includegraphics[width=0.26\textwidth]{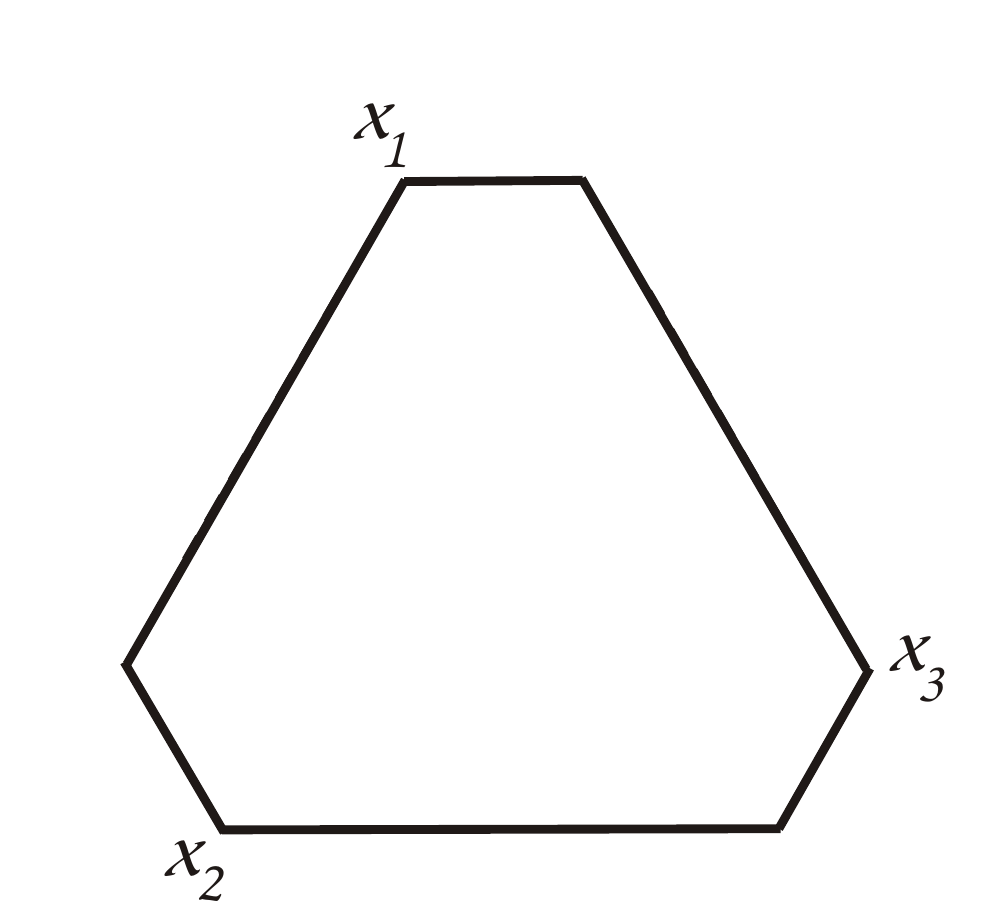}\\
  \caption{A $3$-rotationally symmetric hexagon}\label{fig:cuttri}
\end{figure}

We can now define the notion of Cheeger-regular set in the class of $k$-rotationally symmetric planar convex bodies.

\begin{definition}
\label{def:Cheeger-regular}
Let $\Om$ be a $k$-rotationally symmetric planar convex body, with dots $x_1,\ldots,x_k$ and edges $L_1,\ldots,L_k$.
Let $C_\Om$ be the Cheeger set of $\Om$.
Then $\Om$ is called Cheeger-regular if $C_\Om$ touches all the edges of $\Om$.
\end{definition}

\begin{remark}
Definition~\ref{def:Cheeger-regular} does not depend on the choice of the dots of $\Om$.
\end{remark}

\section{Main results}
\label{sec:main}
In this section we prove our main results, by using Theorem~\ref{th:KL}.
Theorem~\ref{th:main} shows that any $k$-rotationally symmetric planar convex body $\Om$ is Cheeger-regular,
and Corollary~\ref{co:rot} assures that the Cheeger set of $\Om$ is always $k$-rotationally symmetric
(that is, in this setting, the rotational symmetry is inherited by the Cheeger set).

\begin{theorem}
\label{th:main}
Let $\Om$ be a $k$-rotationally symmetric planar convex body.
Then, $\Om$ is Cheeger-regular, that is, its Cheeger set $C_\Om$ touches all the edges of $\Om$.
\end{theorem}

\begin{proof}
Fix $q\in\ptl C_\Om\cap\ptl\Om$ (recall that $\ptl C_\Om$ touches necessarily $\ptl\Om$).
We can assume that $q$ lies in the edge $L_1$ of $\Om$.
Taking into account the characterization of $C_\Om$ from Theorem~\ref{th:KL},
it follows there exists $s>0$ such that $q=z+s\,v$, for certain $z\in\Om^{s}$ and $v\in\sph^1$,
and by applying the rotation $\theta_k$, we will obtain that $\theta_k(q)=\theta_k(z)+s\,\theta_k(v)$.
Observe that $\theta_k(q)\in L_2$ and $\theta_k(v)\in\sph^1$. Let us prove that $\theta_k(z)\in\Om^{s}$.

Since $z\in\Om^{s}$, we have that $dist(z,\ptl\Om)=\min\{d(z,w):w\in\ptl\Om\}>s$.
Clearly, $d(z,w)=d(\theta_k(z),\theta_k(w))$ for any $w\in\ptl\Om$.
Since $\theta_k$ is bijective when restricted to $\ptl\Om$, 
it follows that
$$s<\min\{d(z,w):w\in\ptl\Om\}=\min\{d(\theta_k(z),w):w\in\ptl\Om\},$$
which implies that $\theta_k(z)\in\Om^{s}$.
Therefore $\theta_k(q)=\theta_k(z)+s\,\theta_k(v)\in\Om^{s}+s\,B_1=C_\Om$, using again Theorem~\ref{th:KL}.
This gives that $C_\Om\cap L_2\neq\emptyset$, and repeating the argument,
it yields that $C_\Om\cap L_i\neq\emptyset$, for $i=1,\ldots,k$,
which means that $C_\Om$ touches all the edges of $\Om$, as stated.
\end{proof}

In view of Theorem~\ref{th:main}, we have that the Cheeger set
of a given $k$-rotationally symmetric planar convex body $\Om$
can be obtained by rounding properly the dots of $\Om$.
We also have the following consequence.

\begin{corollary}
\label{co:rot}
Let $\Om$ be a $k$-rotationally symmetric planar convex body.
Then, its Cheeger set $C_\Om$ is $k$-rotationally symmetric.
\end{corollary}

\begin{proof}
It suffices to prove that $\theta_k(C_\Om)\subset C_\Om$.
For any $q\in C_\Om$, we will have that $q=z+s\,v$,
for certain $z\in\Om^{t^*}$, $s>0$ and $w\in B_1$, in view of Theorem~\ref{th:KL}.
Then $\theta_k(q)=\theta_k(z)+s\,\theta_k(v)$.
It is clear that $\theta_k(v)\in B_1$,
and the same reasonings from the proof of Theorem~\ref{th:main} give that $\theta_k(z)\in\Om^{s}$.
Thus $\theta_k(q)\in\Om^{s}+s\,B_1=C_\Om$, which concludes the proof.
\end{proof}

\begin{remark}
The reader may compare our Corollary~\ref{co:rot} with~\cite[Lemma~2.3]{BP},
where it is proved that for a convex bounded domain $\Om$ in $\rr^n$, $n\geq3$,
which is also rotationally invariant
(in the sense that $\Om$ is invariant under the rotation about a fixed axis),
the corresponding Cheeger set inherits the same rotational invariance.
\end{remark}

\section{Further comments}
We finish with some general comments of interest on the Cheeger problem.

\begin{remark}
\label{re:potencias}
An interesting variant of the Cheeger problem, stated in~\cite[Problem~A23]{cfg},
consists of minimizing the quotient $P(X)^\alpha/A(X)$, for $\alpha>0$,
over all subsets $X$ contained in a given planar (convex) body $\Om$, looking also for the optimal sets.
We remark that a related question in general dimension (the \emph{fractional Cheeger problem})
is treated in~\cite{BLP}, and the reciprocal problem of minimizing $P(X)/V(X)^\alpha$
(over subsets $X$ of a given domain $\Om$ in $\rr^n$) is considered in~\cite{fmp},
where a sharp lower bound for the corresponding Cheeger constant is obtained.
\end{remark}

\begin{remark}
We also have, as a sort of extension of the problems described in Remark~\ref{re:potencias},
the so-called \emph{generalized Cheeger problem},
where two different positive functions on a planar bounded domain
are used for weighting the perimeter and area functionals.
This is partially studied in~\cite{IL},
with interesting applications in the framework of landslides modelling,
since the generalized Cheeger constant coincides with a certain coefficient
related to the stability of landslides (see~\cite{IL} and references therein).
In that work, among other things, the Cheeger sets and Cheeger constants for this generalized problem
are described when considering planar rectangles and a particular family of weights
which depend only on one variable~\cite[Th.~7 and Re.~5.E]{IL}.
These results may suggest that the study of this generalized Cheeger problem
in different domains and with different weights
(also known as \emph{densities}, which have been deeply studied in the last years, see~\cite{morgan})
is an interesting question which could reveal applications in other settings
(for instance, there is a strong connection with the existence of
rotationally invariant Cheeger sets in domains of $\rr^n$, see~\cite[Re.~5.1]{BP}).
We further point out that the generalized Cheeger problem in $\rr^n$ with the Gaussian density
weighting the perimeter and volume functionals has been considered in~\cite{CMN}.
\end{remark}

\begin{remark}
An additional related question is posing the Cheeger problem for the other classical geometric magnitudes,
that is, minimizing the quotient $F(X)/A(X)$ over all subsets $X$ contained in a given planar convex body $\Om$,
where $F(X)$ stands for the \emph{minimal width}, or the \emph{circumradius}, or the \emph{inradius} of $X$.
This was also suggested in~\cite[Problem~A.23]{cfg}, and we have not found any related reference in literature.
This is surprising for us, because the relations between pairs (and triplets)
of these classical magnitudes have been largely studied for a long time~\cite{SA}.
\end{remark}

\begin{remark}
Among all the polygons of (at most) $n$ sides enclosing a prescribed quantity of area,
it is known that the regular one provides the minimum possible value for the Cheeger constant~\cite{BF}.
This result, whose stability complement can be found in~\cite{CN},
is strongly related to a conjecture by P\'olya and Szeg\H o
on the polygon (of $n$ sides and prescribed enclosed area)
minimizing the first eigenvalue of the Laplacian~\cite{PS}.
\end{remark}

\begin{remark}
In the non-convex planar setting, there are only few situations where the Cheeger problem has been succesfully considered.
For a planar annulus, the Cheeger set is the annulus itself, and some advances have been obtained for
non-convex planar strips~\cite{KP,LP} (see also~\cite{KLV} for some similar recent results in general dimension).
Recall that when the set is not convex, uniqueness is not assured, as shown in~\cite{KL} describing a particular example.
\end{remark}

\begin{remark}
As previously commented, the Cheeger problem has been well studied
during the last decade in some families of sets with certain characteristics
(for instance, planar curved strips~\cite{KP,LP}, or rotationally symmetric domains~\cite{BP}).
This reveals the importance of determining appropriately the setting for studying this problem.
\end{remark}

{\bf Acknowledgments.} The author is partially supported by the MICINN project MTM2017-84851-C2-1-P, 
and by Junta de Andaluc\'ia grant FQM-325. The author would like to thank G.~P.~Leonardi for his help during the elaboration of this note.

\end{document}